\documentclass[12pt]{article}
\usepackage{amssymb}
\usepackage{amsmath,amsthm}
\usepackage[latin1]{inputenc}
\usepackage{hyperref}
\usepackage{enumerate}
\usepackage{color}
\usepackage{graphicx}
\hypersetup{colorlinks=true, linkcolor=blue, citecolor=blue,urlcolor=blue}

 \newtheorem{remark}{Remark}

 \newtheorem{lemma}[remark]{Lemma}
 \newtheorem{theorem}[remark]{Theorem}
 \newtheorem{proposition}[remark]{Proposition}
 \newtheorem{corollary}[remark]{Corollary}

\title{Strong metric dimension of rooted product graphs}

\date{}

\author{  Dorota Kuziak$^{(1)}$, Ismael G. Yero$^{(2)}$ and Juan A.
Rodr\'{\i}guez-Vel\'{a}zquez$^{(1)}$
    \\
$^{(1)}${\small Departament d'Enginyeria Inform\`atica i Matem\`atiques,}\\
{\small Universitat Rovira i Virgili,}  {\small Av. Pa\"{\i}sos
Catalans 26, 43007 Tarragona, Spain.} \\{\small
juanalberto.rodriguez\@@urv.cat, dorota.kuziak\@@urv.cat}
\\
$^{(2)}${\small Departamento de Matem\'aticas, Escuela Polit\'ecnica Superior de Algeciras}\\
{\small Universidad de C\'adiz,} {\small
Av. Ram\'on Puyol s/n, 11202 Algeciras, Spain.} \\ {\small
ismael.gonzalez\@@uca.es}\\
}

\begin{document}

\maketitle

\begin{abstract}
Let $G$ be a connected graph. A vertex $w$  strongly resolves a pair $u$, $v$ of vertices of $G$ if there exists some shortest $u-w$ path containing $v$ or some shortest $v-w$ path containing $u$. A set $W$ of vertices is a strong resolving set for $G$ if every pair of
vertices of $G$ is strongly resolved by some vertex of $W$. The smallest cardinality of a strong resolving set for $G$ is called the strong metric
dimension of $G$. It is known that the problem of computing this invariant is NP-hard. This suggests finding the strong metric dimension for special classes of graphs or obtaining good bounds on this invariant.  In this paper we study the problem of finding exact values or sharp bounds for the strong metric dimension of rooted product of graphs and express these in terms of invariants of the factor graphs.
\end{abstract}

{\it Keywords:} Strong metric dimension; rooted product graphs; strong metric basis; strong resolving set.

{\it AMS Subject Classification Numbers:}  05C12; 05C69; 05C76.

\section{Introduction}

A {\em generator} of a metric space is a set $S$ of points in the space with the property that every point of the space is uniquely determined by its distances from the elements of $S$. Given a simple and connected graph $G=(V,E)$, we consider the metric $d_G:V\times V\rightarrow \mathbb{R}^+$, where $d_G(x,y)$ is the length of a shortest path between $x$ and $y$. $(V,d_G)$ is clearly a metric space. A vertex $v\in V$ is said to distinguish two vertices $x$ and $y$ if $d_G(v,x)\ne d_G(v,y)$.
A set $S\subset V$ is said to be a \emph{metric generator} for $G$ if any pair of vertices of $G$ is
distinguished by some element of $S$. A minimum generator is called a \emph{metric basis}, and
its cardinality the \emph{metric dimension} of $G$, denoted by $dim(G)$. Motivated by the problem of uniquely determining the location of an intruder in a network, the concept of metric
dimension of a graph was introduced by Slater in \cite{leaves-trees,slater2}, where the metric generators were called \emph{locating sets}. The concept of metric dimension of a graph was introduced independently by Harary and Melter in \cite{harary}, where metric generators were called \emph{resolving sets}. Applications of this invariant to the navigation of robots in networks are discussed in \cite{landmarks} and applications to chemistry in \cite{pharmacy1,pharmacy2}.  This invariant was studied further in a number of other papers including for example, \cite{pelayo1, chappell,chartrand,chartrand2, fehr,haynes,Tomescu1,LocalMetric,survey,tomescu,yerocartpartres,CMWA,metric-rooted}.  Several variations of metric generators including resolving dominating sets \cite{brigham}, independent resolving sets \cite{chartrand1}, local metric sets \cite{LocalMetric}, and strong resolving sets \cite{strongDimensionCorona,Oellermann,seb}, etc. have been introduced and studied.

In this article we are interested in the study  of strong resolving sets \cite{Oellermann,seb}.
A vertex $w\in V(G)$ \emph{strongly resolves} two vertices $u,v\in V(G)$ if 
$d_G(w,u)=d_G(w,v)+d_G(v,u)$ or $d_G(w,v)=d_G(w,u)+d_G(u,v)$, i.e., there exists some shortest $w-u$ path containing $v$ or some shortest $w-v$ path containing $u$. A set $S$ of vertices in a connected graph $G$ is a \emph{strong metric generator} for $G$ if every two vertices of $G$ are strongly resolved by some vertex
of $S$. The smallest cardinality of a strong resolving set of $G$ is called \emph{strong metric dimension} and is denoted by $dim_s(G)$.  So, for example, $dim_s(G)=n-1$ if and only if $G$ is the complete graph  of order $n$.
For the cycle $C_n$ of order $n$ the strong metric dimension is $dim_s(C_n) = \lceil n/2\rceil$  and if $T$ is a tree with $l(T)$ leaves, its strong metric dimension equals  $l(T)-1$   (see \cite{seb}).  A \emph{strong metric basis} of $G$ is a strong metric generator for $G$ of cardinality $dim_s(G)$.

Given a simple graph $G=(V,E)$, we denote two adjacent vertices $u,v$ by $u\sim v$. The \emph{neighborhood of a vertex} $v$ of $G$ is $N_G(v)=\{u\in V(G): u\sim v\}$ and the degree of $v$ is $\delta_G(v)=|N_G(v)|$. The \emph{open neighborhood of a set} $S$ of vertices of $G$ is $N_G(S)=\bigcup_{v\in S}N_G(v)$ and the \emph{closed neighborhood of} $S$ is $N_G[S]=N_G(S)\cup S$. The \emph{subgraph induced by a set} $X$ will be denoted by $\langle X \rangle$.
A vertex $u$ of $G$ is \emph{maximally distant} from $v$ if for every vertex $w$ in the open neighborhood of $u$, $d_G(v,w)\le d_G(u,v)$. If $u$ is maximally distant from $v$ and $v$ is maximally distant from $u$, then we say that $u$ and $v$ are \emph{mutually maximally distant}.  The {\em boundary} of $G=(V,E)$ is defined as $\partial(G) = \{u\in V:$ there exists $v\in V$  such that $u,v$  are mutually maximally distant$\}$. For some basic graph classes, such as complete graphs $K_n$, complete bipartite graphs $K_{r,s}$,  cycles $C_n$ and hypercube graphs $Q_k$, the boundary is simply the whole vertex set.
It is not difficult to see that this property holds for all  $2$-antipodal\footnote{The diameter of $G=(V,E)$ is defined as $D(G)=\max_{u,v\in V}\{d(u,v)\}$.  We recall that $G=(V,E)$ is $2$-antipodal if for each vertex $x\in V$ there exists exactly one vertex $y\in V$ such that $d_G(x,y)=D(G)$.} graphs and also for all distance-regular graphs.
Notice that the boundary of a tree consists exactly of the set of its leaves. A vertex  of a graph  is a {\em simplicial vertex} if the subgraph induced by  its neighbors is a complete graph. Given a graph $G$, we denote by $\sigma(G)$ the set of simplicial vertices of $G$.
Notice that $\sigma(G)\subseteq \partial(G)$.

We use the notion of strong resolving graph introduced in \cite{Oellermann}. The \emph{strong resolving graph}\footnote{In fact, according to \cite{Oellermann} the strong resolving graph $G'_{SR}$ of a graph $G$ has vertex set $V(G'_{SR})=V(G)$ and two vertices $u,v$ are adjacent in $G'_{SR}$ if and only if $u$ and $v$ are mutually maximally distant in $G$. So, the strong resolving graph defined here is a subgraph of the strong resolving graph defined in \cite{Oellermann} and it can be obtained from the latter graph by deleting its isolated vertices.} of $G$ is a graph $G_{SR}$  with vertex set $V(G_{SR}) = \partial(G)$ where two vertices $u,v$ are adjacent in $G_{SR}$ if and only if $u$ and $v$ are mutually maximally distant in $G$.

There are some families of graphs for which its strong resolving graph can be obtained relatively easy. For instance, we emphasize the following cases.
\begin{itemize}
\item If $\partial(G)=\sigma(G)$, then $G_{SR}\cong K_{|\partial(G)|}$. In particular, $(K_n)_{SR}\cong K_n$ and for any tree $T$ with $l(T)$ leaves, $(T)_{SR}\cong K_{l(T)}$.

\item For any $2$-antipodal graph $G$ of order $n$, $G_{SR}\cong \bigcup_{i=1}^{\frac{n}{2}} K_2$. In particular,  $(C_{2k})_{SR}\cong \bigcup_{i=1}^{k} K_2$.

\item $(C_{2k+1})_{SR}\cong C_{2k+1}$.
\end{itemize}

A set $S$ of vertices of $G$ is a \emph{vertex cover} of $G$ if every edge of $G$ is incident with at least one vertex of $S$. The \emph{vertex cover number} of $G$, denoted by $\alpha(G)$, is the smallest cardinality of a vertex cover of $G$. We refer to an $\alpha(G)$-set in a graph $G$ as a vertex cover  of cardinality $\alpha(G)$. Oellermann and Peters-Fransen \cite{Oellermann} showed that the problem of finding the strong metric dimension of a connected graph $G$ can be transformed to the problem of finding the vertex cover number of $G_{SR}$. The following result  will be an important tool of this article.

\begin{theorem}{\em \cite{Oellermann}}\label{lem_oellerman}
Let $G$ be connected graph. A set $W\subset V(G)$ is a strong metric generator for $G$ if and only if $W$ is a vertex cover for $G_{SR}$.
\end{theorem}

It was shown in \cite{Oellermann}  that the problem of computing $dim_s(G)$ is NP-hard. This suggests finding the strong metric dimension for special classes of graphs or obtaining good bounds on this invariant. An efficient procedure for finding the strong metric dimension of distance hereditary graphs was described in \cite{may-oellermann}. In this paper we study the problem of finding exact values or sharp bounds for the strong metric dimension of rooted product of graphs and express these in terms of invariants of the factor graphs. Notice that the metric dimension of rooted product graphs has been recently studied in \cite{metric-rooted}.


\begin{figure}[h]
  \centering
  \includegraphics[width=0.3\textwidth]{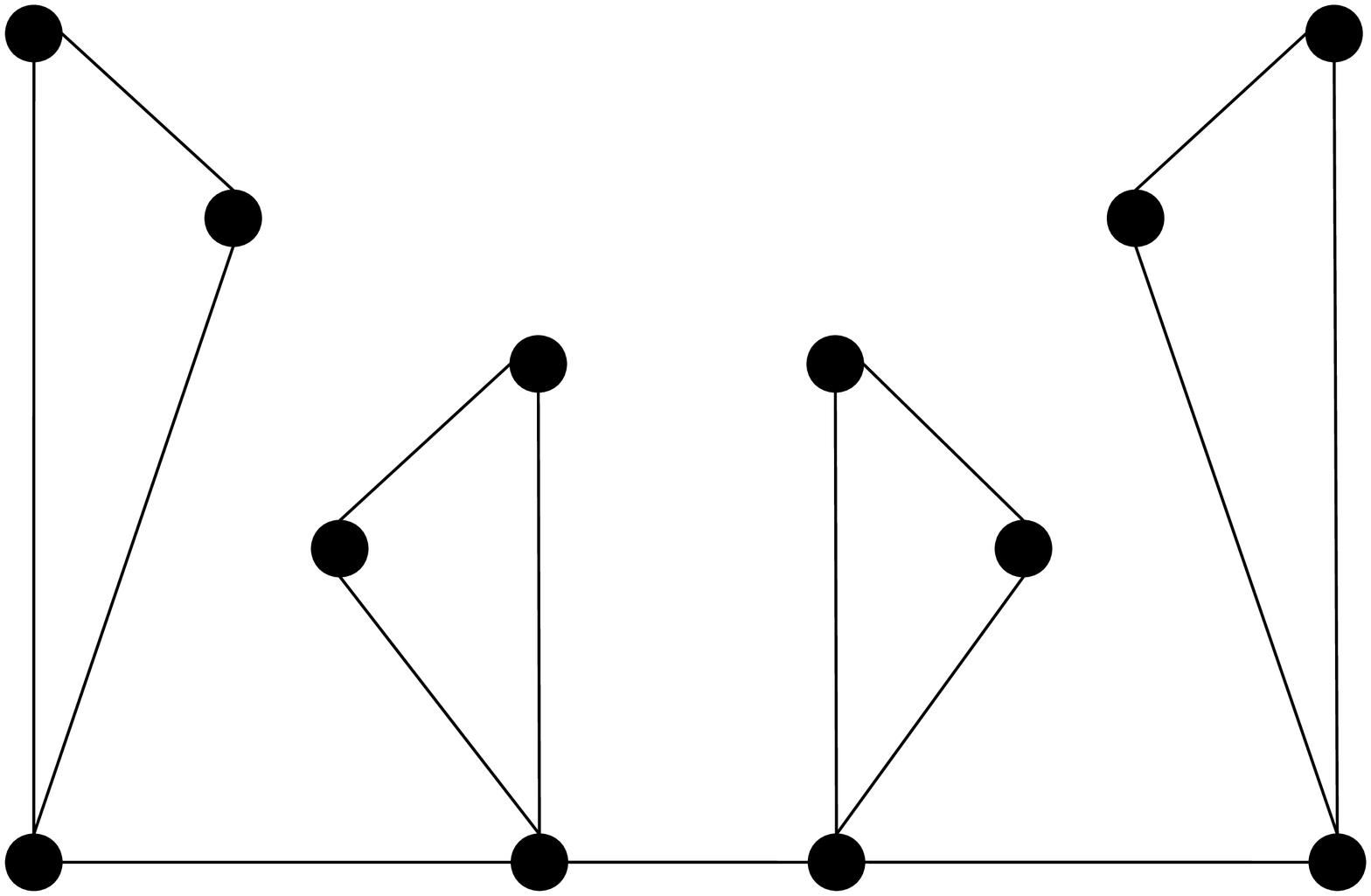}
  \hspace*{2.0cm} \includegraphics[width=0.3\textwidth]{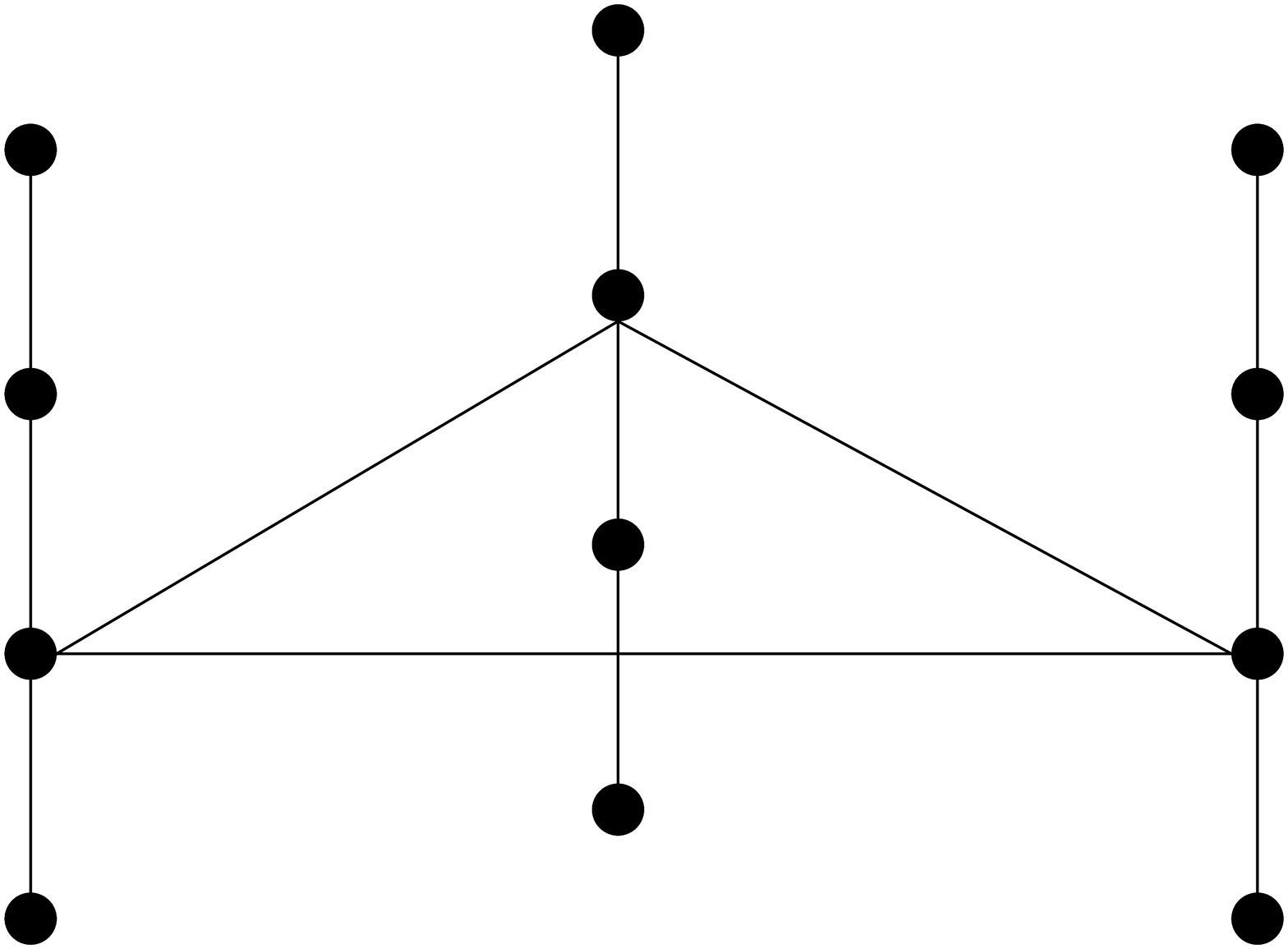}
  \caption{The rooted product graphs $P_4\circ C_3$ and $C_3\circ_v P_4$, where $v$ has degree two.}\label{p-4-c-3}
\end{figure}

A \emph{rooted graph} is a graph in which one vertex is labeled in a special way so as to distinguish it from other vertices. The special vertex is called the \emph{root} of the graph. Let $G$ be a labeled graph on $n$ vertices. Let ${\cal H}$ be a sequence of $n$ rooted graphs $H_1$, $H_2$,...,$H_n$. The \emph{rooted product graph} $G( {\cal H})$ is the graph obtained by identifying the root of $H_i$ with the $i^{th}$ vertex of $G$ \cite{rooted-first}. In this paper we consider the particular case of rooted product graph where ${\cal H}$ consists of $n$  isomorphic rooted graphs \cite{Schwenk}. More formally, assuming that $V(G) = \{u_1, ..., u_n\}$ and that the root vertex of $H$ is $v$, we define the rooted product graph $G\circ_{v} H=(V,E)$, where $V=V(G)\times V(H)$ and
$$E=\displaystyle\bigcup_{i=1 }^n\{(u_i,b)(u_i,y): \; by\in E(H)\}\cup \{(u_i,v)(u_j,v):\; u_iu_j\in E(G)\}.$$ Note that for any $x\in V(G)$ the subgraph $H_x=\langle\{x\}\times V(H)\rangle$ of $G\circ_v H$ is isomorphic to $H$. Given $x\in V(G)$, $v\in V(H)$ and  $B\subset V(G)\times V(H)$ we will denote by $B_x$ the set of element of $B$ whose first component is $x$, i.e., $B_x=B\cap (\{x\}\times V(H))$.

If $H$ is a vertex transitive graph, then $G\circ_v H$ does not depend on the choice of $v$, up to isomorphism. In such a case  we will denote the rooted product by $G\circ H$. Figure \ref{p-4-c-3} shows the case of the rooted product graphs $P_4\circ C_3$ and $C_3\circ_v P_4$, where $v$ has degree two. We also recall that the \emph{corona product} $G\odot H$ is defined as the graph obtained from $G$ and $H$ by taking one copy of $G$ and $n$ copies of $H$ and joining by an edge each vertex from the $i^{th}$-copy of $H$ with the $i^{th}$-vertex of $G$. If $G$ and $H$ are connected graphs of order $n\ge 2$, $H$ is a connected graph of order $t\ge 2$, then we note that the corona product graph $G\odot H$ is a particular case of a rooted product graph, {\em i.e.}, $G\odot H \cong G\circ_v (K_1+H) $, where $v$ denotes the vertex of $K_1$. Metric dimension and strong metric dimension of corona product graphs were studied in \cite{CMWA} and \cite{strongDimensionCorona}, respectively.

We emphasize that given $a,b\in V(G)$ and $x,y,v\in V(H)$ it follows, $d_{G\circ_{v} H}((a,x),(a,y))=d_H(x,y)$ and if $a\ne b$, then $d_{G\circ_{v} H}((a,x),(b,y))=d_H(x,v)+d_G(a,b)+d_H(v,y)$.

This article is composed by two main sections. In Section \ref{Formulae}  we obtain closed formulae for the strong metric dimension of some classes of rooted product graphs while  Section \ref{bounds} is devoted to obtain tight bounds for the strong metric dimension of rooted product graphs.

\section{Closed formulae}\label{Formulae}

We start by stating the following easily verified lemmas.

\begin{lemma}\label{maximally_distant with root}
Let $G$ and $H$ be two connected  graphs. Let $a,b\in V(G)$, $a\ne b$, $x,y,v\in V(H)$ and let $M(v)$ be the set of vertices of $ H$ which are maximally distant from $v$. Then $(a,x)$ and $(b,y)$ are mutually maximally distant vertices in $G\circ_v H$ if and only if $x,y\in M(v)$.
\end{lemma}

\begin{proof}
(Sufficiency) Suppose that $(a,x)$ and $(b,y)$ are not mutually maximally distant vertices in $G\circ_{v}H$. So, there exists a vertex $(a,x')\in N_{G\circ_{v} H}(a,x)$ such that $d_{G\circ_{v} H}((a,x'),(b,y))>d_{G\circ_{v} H}((a,x),(b,y))$, or there exists $(b,y')\in N_{G\circ_{v} H}(b,y)$ such that $d_{G\circ_{v} H}((a,x),(b,y'))>d_{G\circ_{v} H}((a,x),(b,y))$. We consider, without loss of generality, that $(a,x')\in N_{G\circ_{v} H}(a,x)$ and $d_{G\circ_{v} H}((a,x'),(b,y))>d_{G\circ_{v} H}((a,x),(b,y))$. So we have,
\begin{align*}
d_H(x',v)&=d_{G\circ_{v} H}((a,x'),(b,y))-d_G(a,b)-d_H(v,y)\\
&>d_{G\circ_{v} H}((a,x),(b,y))-d_G(a,b)-d_H(v,y)\\
&=d_H(x,v).
\end{align*}
Thus, $d_H(x',v)>d_H(x,v)$. Since $x'\in N_H(x)$ and $x\in M(v)$, we have a contradiction.

(Necessity) Let us suppose that $x\notin M(v)$. So, there exists $x''\in N_H(x)$ such that $d_H(x'',v)>d_H(x,v)$. Thus, $d_{G\circ_{v} H}((a,x),(b,y))=d_H(x,v)+d_G(a,b)+d_H(v,y)<d_H(x'',v)+d_G(a,b)+d_H(v,y)=d_{G\circ_{v} H}((a,x''),(b,y))$. Hence, there exists a vertex $(a,x'')\in N_{G\circ_{v} H}((a,x))$ such that $d_{G\circ_{v} H}((a,x),(b,y))<d_{G\circ_{v} H}((a,x''),(b,y))$, which is a contradiction since $(a,x)$ and $(b,y)$ are mutually maximally distant.\end{proof}

\begin{lemma}\label{rooted mutually_maximally_distant}
Let $G$ and $H$ be two connected nontrivial graphs. Let $v,x,y$ be vertices of $H$ such that $x,y\ne v$. For every vertex $a$ of $G$ we have that $(a,x)$ and $(a,y)$ are mutually maximally distant vertices in $G\circ_v H$ if and only if the vertices $x$ and $y$ are mutually maximally distant in $H$.
\end{lemma}

\begin{proof}
The result follows directly from the fact that for every vertex $c$ of $G$ and every vertex $z\ne v$ of $H$ we have that $w\in N_H(z)$ if and only if $(c,w)\in N_{G\circ_{v} H}(c,z)$ and also that $d _{G\circ_v H}((a,x),(a,y))=d_H(x,y)$ for every $x,y$ of $H$.
\end{proof}

\begin{lemma}\label{maximally distant with v}
Let $H$ be a connected graph, let $v\in V(H)$ and let $M(v)$ be the set of vertices of $ H$ which are maximally distant from $v$. Then $M(v)\subseteq \partial(H)$.
\end{lemma}

\begin{proof}
Let $u\in M(v)$. If $v$ is not maximally distant from $u$, then there exists a vertex $y_1\in N(v)$ such that $d(y_1,u)>d(v,u)$. So $u$ is maximally distant from $y_1$. By repeating this argument, since $H$ is finite, we will find a vertex $y_i$ such that $y_i$ and $u$ are mutually maximally distant. Therefore, $u\in \partial(H)$.
\end{proof}

\begin{proposition}\label{boundary-of-rooted}
Let $G$ be a connected graph of order $n\ge 2$ and let $H$  be a connected graph.
\begin{enumerate}[{\rm (i)}]
\item If   $v\in \partial(H)$, then  $\partial \left(G\circ_{v} H\right)=V(G)\times (\partial(H)-\{v\}).$

\item If   $v\not\in \partial(H)$, then $\partial \left(G\circ_{v} H\right)=V(G)\times \partial(H).$
\end{enumerate}
\end{proposition}

\begin{proof}
Let $(x,y)$ and $(x',y')$ be two mutually maximally distant vertices in $G\circ_{v} H$.  Since  $(V(G)\times \{v\})\cap \partial \left(G\circ_{v} H\right)=\emptyset$, it follows $y,y'\ne v$. We differentiate two cases.

Case 1:  $x=x'$. By Lemma \ref{rooted mutually_maximally_distant}
   we conclude that $(x,y)$ and $(x',y')$ are mutually maximally distant in $G\circ_{v} H$ if and only if  $y$ and $y'$ are mutually maximally distant  in $H$.

Case 2: $x\ne x'$. By Lemma \ref{maximally_distant with root} the vertices $(x,y)$ and $(x',y')$ are mutually maximally distant in $G\circ_{v} H$ if and only if  $y,y'\in M(v)$. Note that, by Lemma \ref{maximally distant with v},  $y,y'\in \partial(H)$.

According to the above cases we conclude that if $(x,y)\in \partial(G\circ_{v} H)$, then $y\in \partial(H)-\{v\}$. Moreover, if $y\in \partial(H)-\{v\}$, then for every $x\in V(G)$ we have $(x,y)\in \partial(G\circ_{v} H)$.

Therefore, if  $v\in \partial(H)$, then   $\partial \left(G\circ_{v} H\right)=V(G)\times (\partial(H)-\{v\})$ and if
 $v\not \in \partial(H)$, then  $\partial \left(G\circ_{v} H\right)=V(G)\times \partial(H).$
\end{proof}

\begin{proposition}\label{simplicial-of-rooted}
Let $G$ be a connected graph of order $n\ge 2$ and let $H$  be a connected graph.
\begin{enumerate}[{\rm (i)}]
\item If   $v\in \sigma(H)$, then  $\sigma \left(G\circ_{v} H\right)=V(G)\times (\sigma(H)-\{v\}).$

\item If   $v\not\in \sigma(H)$, then $\sigma \left(G\circ_{v} H\right)=V(G)\times \sigma(H).$
\end{enumerate}
\end{proposition}

\begin{proof}
Note that   $(x,v)$ is not simplicial  in $G\circ_{v} H$. Since the following  assertions are equivalent, the result immediately follows.
 \begin{itemize}
 \item  The vertex $(x,y)\in V(G)\times (V(H)-\{v\})$ is simplicial in $G\circ_{v} H$.
  \item For $x\in V(G)$ and $y\ne v$ the vertex $(x,y)$ is simplicial in $H_x$.
  \item The vertex $y\in V(H)-\{v\}$ is simplicial in $H$.
\end{itemize}
\end{proof}

\begin{theorem}\label{ThCompleteSR}
Let $G$ be a connected graph of order $n\ge 2$ and let $H$  be a connected graph such that  $\partial(H)=\sigma(H)$.
\begin{enumerate}[{\rm (i)}]
  \item If   $v\in \partial(H)$, then
$dim_s(G\circ_{v} H) = n(|\partial(H)|-1)-1.$
\item  If   $v\not\in \partial(H)$, then
$dim_s(G\circ_{v} H) = n|\partial(H)|-1.$
\end{enumerate}
\end{theorem}
\begin{proof}
  Since $\partial(H)=\sigma(H)$, as a direct consequence of Proposition \ref{boundary-of-rooted} and Proposition \ref{simplicial-of-rooted} we obtain that if $v\not\in \partial(H)$, then $\partial \left(G\circ_{v} H\right)=V(G)\times \partial(H)=\sigma\left(G\circ_{v} H\right)$ and if $v\in \partial(H)$, then $\partial \left(G\circ_{v} H\right)=V(G)\times (\partial(H)-\{v\})=\sigma\left(G\circ_{v} H\right).$ Hence, if
  $v\not\in \partial(H)$, then $(G\circ_v H)_{SR}\cong K_{n|\partial(H)|}$ and, if  $v\in \partial(H)$, then $(G\circ_v H)_{SR}\cong K_{n(|\partial(H)|-1)}$. Therefore, the result follows by Theorem \ref{lem_oellerman}.
\end{proof}

We emphasize the following particular cases of Theorem \ref{ThCompleteSR}.
\begin{corollary}\label{corollari-simplicial}
  Let $G$ be a connected graph of order $n\ge 2$.
  \begin{enumerate}[{\rm (i)}]
  \item For any complete graph of order $n'$, $dim_s(G\circ  K_{n'}) = n(n'-1)-1.$
  \item For any  tree $T$  with $l(T)$ leaves,
            $$dim_s(G\circ_{v} T)=\left\{\begin{array}{ll}

                            n(l(T)-1)-1, & \mbox{if $v$ is a leaf of $T$,} \\
                             & \\
                             n\cdot l(T)-1, & \mbox{if $v$ is an inner vertex of $T$.}
                             \end{array} \right.
            $$
  \item Let $G'$ be a connected graph of order $n'$ and let $H=G'\odot (\bigcup_{i=1}^r K_{t_i})$, where $r\ge 2$, $t_i\ge 1$. Then

  $$dim_s(G\circ_{v} H)=\left\{\begin{array}{ll}
                            n\sum_{i=1}^r t_i-n -1, & \mbox{if $v\in \bigcup_{i=1}^r V(K_{t_i})$,} \\
                             & \\
                             n\sum_{i=1}^r t_i -1, & \mbox{if $v\in V(G')$.}
                             \end{array} \right.
            $$
  \end{enumerate}
\end{corollary}

\begin{theorem}\label{ThK_2SR}
Let $G$ be a connected graph of order $n\ge 2$ and let $H$  be a connected graph such that  $H_{SR}\cong \displaystyle\bigcup_{i=1}^{\frac{|\partial(H)|}{2}}K_2$. Let $v\in V(H)$ and let $M(v)$ be the set of vertices of $H$ which are maximally distant from $v$. Let $i(v)$ be the set of isolated vertices of the subgraph of $H_{SR}$ induced by $M(v)$.
\begin{enumerate}[{\rm (i)}]
  \item If   $v\not\in \partial(H)$, then
$dim_s(G\circ_{v} H) = \displaystyle\frac{n(|\partial(H)|+|M(v)|-|i(v)|)-|M(v)|+|i(v)|}{2}$.
\item  If   $v\in \partial(H)$,
then $dim_s(G\circ_{v} H) = \displaystyle\frac{n(|\partial(H)|+|M(v)|-|i(v)|)-|M(v)|+|i(v)|-2}{2}.$
\end{enumerate}
\end{theorem}

\begin{proof}
Let $V(G)=\{x_1,x_2,...,x_n\}$ be the vertex set of $G$ and let $B$ be a vertex cover for $(G\circ_v H)_{SR}$. First we note that by premiss for every $a\in \partial(H)$ there exists exactly one vertex $a'\in \partial(H)$ such that $a$ and $a'$ are adjacent in $H_{SR}$.  We consider the set $i'(v)\subset \partial(H)$  defined in the following way: $a'\in i'(v)$ if and only if there exists $a\in i(v)$ such that $a$ and $a'$ are mutually maximally distant in $H$. Note that $|i'(v)|=|i(v)|$ and, if $v\in \partial(H)$ and $v,v'$ are mutually maximally distant, then $v\in i'(v)$ and $v'\in i(v)$.
Also, since there are no edges in  $H_{SR}$ connecting vertices belonging to  $M(v)\cup i'(v)$ to vertices belonging to $\partial(H)-M(v)\cup i'(v)$, by Lemmas \ref{maximally_distant with root} and \ref{rooted mutually_maximally_distant} we conclude that there are no edges in $(G\circ_vH)_{SR}$ connecting vertices belonging to $V(G)\times (\partial(H)-(M(v)\cup i'(v))$ to vertices belonging to $V(G)\times (M(v)\cup i'(v))$. With this idea in mind, we proceed to prove the results.

In order to prove (i) we consider that $v\not\in \partial(H)$. Note that in this case by Proposition \ref{boundary-of-rooted} (ii),  $\partial(G\circ_v H)=V(G)\times \partial(H)$.
By Lemma \ref{rooted mutually_maximally_distant}  we have that for every mutually maximally distant vertices $a,a'\in \partial(H)-(M(v)\cup i'(v))$ and every $j\in \{1,...,n\}$ the vertices  $(x_j,a)$ and $(x_j,a')$ are mutually maximally distant in $G\circ_{v} H$ and, as a consequence,  $(x_j,a)\not\in B$ if and only if $(x_j,a')\in B$. Thus, the subgraph of $(G\circ_{v} H)_{SR}$ induced by $V(G)\times (\partial(H)-M(v)\cup i'(v))$ is composed by $\frac{n}{2}(|\partial(H)|-|M(v)|-|i'(v)|)$ components isomorphic to $K_2$.

On the other hand, by Lemma \ref{maximally_distant with root} we have that $(x_j,a), (x_k,a)$ are mutually maximally distant in $G\circ_{v} H$, for every $a\in M(v)$ and $j\ne k$. Thus, if $(x_j,a)\not\in B$ for some $j$, then $(x_k,a)\in B$ for every $k\ne j$. Moreover, as above, Lemma \ref{rooted mutually_maximally_distant} allows us to conclude that given two mutually maximally distant vertices $a,a'\in M(v) \cup i'(v)$ it follows that
 $(x_j,a)\not\in B$ if and only if $(x_j,a')\in B$.   Thus, $B$ contains exactly $(n-1)|M(v)+\frac{|M(v)\cup i'(v)|}{2}$ vertices belonging to $V(G)\times (M(v)\cup i'(v))$. Therefore,
\begin{align*}|B|&= \frac{n(|\partial(H)|-|M(v)|-|i(v)|)}{2} + (n-1)|M(v)|+\frac{|M(v)|+|i(v)|}{2} \\&= \displaystyle\frac{n(|\partial(H)|+|M(v)|-|i(v)|)-|M(v)|+|i(v)|}{2}\end{align*} The proof of (i) is complete.

From now on we suppose $v\in \partial(H)$. Note that in this case by Proposition \ref{boundary-of-rooted} (i) we have  $\partial(G\circ_v H)=V(G)\times (\partial(H)-\{v\})$.
To prove (ii) we proceed by analogy  to the proof of (i). In this case we obtain that  the subgraph of $(G\circ_{v} H)_{SR}$ induced by $V(G)\times (\partial(H)-(M(v)\cup i'(v))$ is composed by $\frac{n}{2}|\partial(H)-M(v)\cup i'(v)|=\frac{n}{2}(|\partial(H)|-|M(v)|-|i(v)|)$ components isomorphic to $K_2$ and $B$ contains exactly $(n-1)|M(v)|+\frac{|(M(v)-\{v'\})\cup (i'(v)-\{v\})|}{2}=(n-1)|M(v)|+\frac{|M(v)|+|i(v)|-2}{2}$ vertices of $G\circ_v H$  belonging to $V(G)\times (M(v)\cup (i'(v)-\{v\}))$. Thus,
\begin{align*}|B|&= \frac{n(|\partial(H)|-|M(v)|-|i(v)|)}{2} + (n-1)|M(v)|+\frac{|M(v)|+|i(v)|-2}{2} \\&= \displaystyle\frac{n(|\partial(H)|+|M(v)|-|i(v)|)-|M(v)|+|i(v)|-2}{2}.\end{align*}
The proof of (ii) is complete.
\end{proof}

We conjecture that if $v\not \in \partial(H)$, then $i(v)=i'(v)=\emptyset$. In order to show a particular case of Theorem \ref{ThK_2SR} where $i(v)\ne \emptyset$ we consider the graph $H$  shown in the left hand side of Figure \ref{grafos-i(v)} where $\partial(H)=\{a,a',b,b',v,v'\}$, $M(v)=i(v)=\{a,v'\}$ and $i'(v)=\{a',v\}$. 
In the case of the graph $H$  shown in the right hand side of Figure \ref{grafos-i(v)} we have $\partial(H)=\{a,a',b,b',v,v'\}$, $M(v)=\{a,a',v'\}$, $i(v)=\{v'\}$ and $i'(v)=\{v\}$. In both cases $$B=(V(G)-\{u_n\})\times (M(v)\cup \{b\})\cup \{(u_n,a),(u_n,b)\}$$ is a strong metric basis of $G\circ_v H$ for any graph $G$ with vertex set $V=\{u_1,u_2,...,u_n\}$.

\begin{figure}[h]
  \centering
  \includegraphics[width=0.4\textwidth]{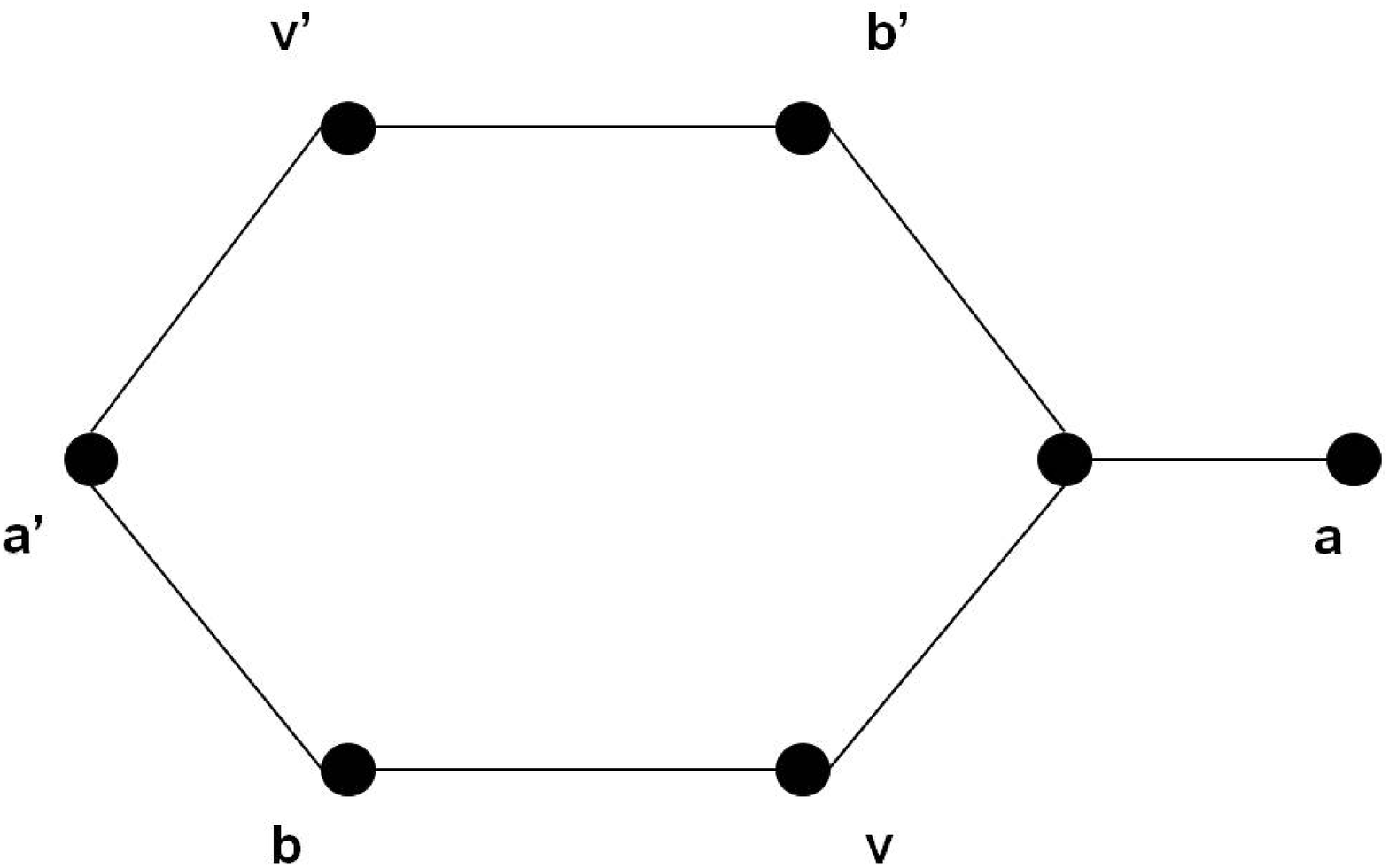}
  \hspace*{2.0cm} \includegraphics[width=0.4\textwidth]{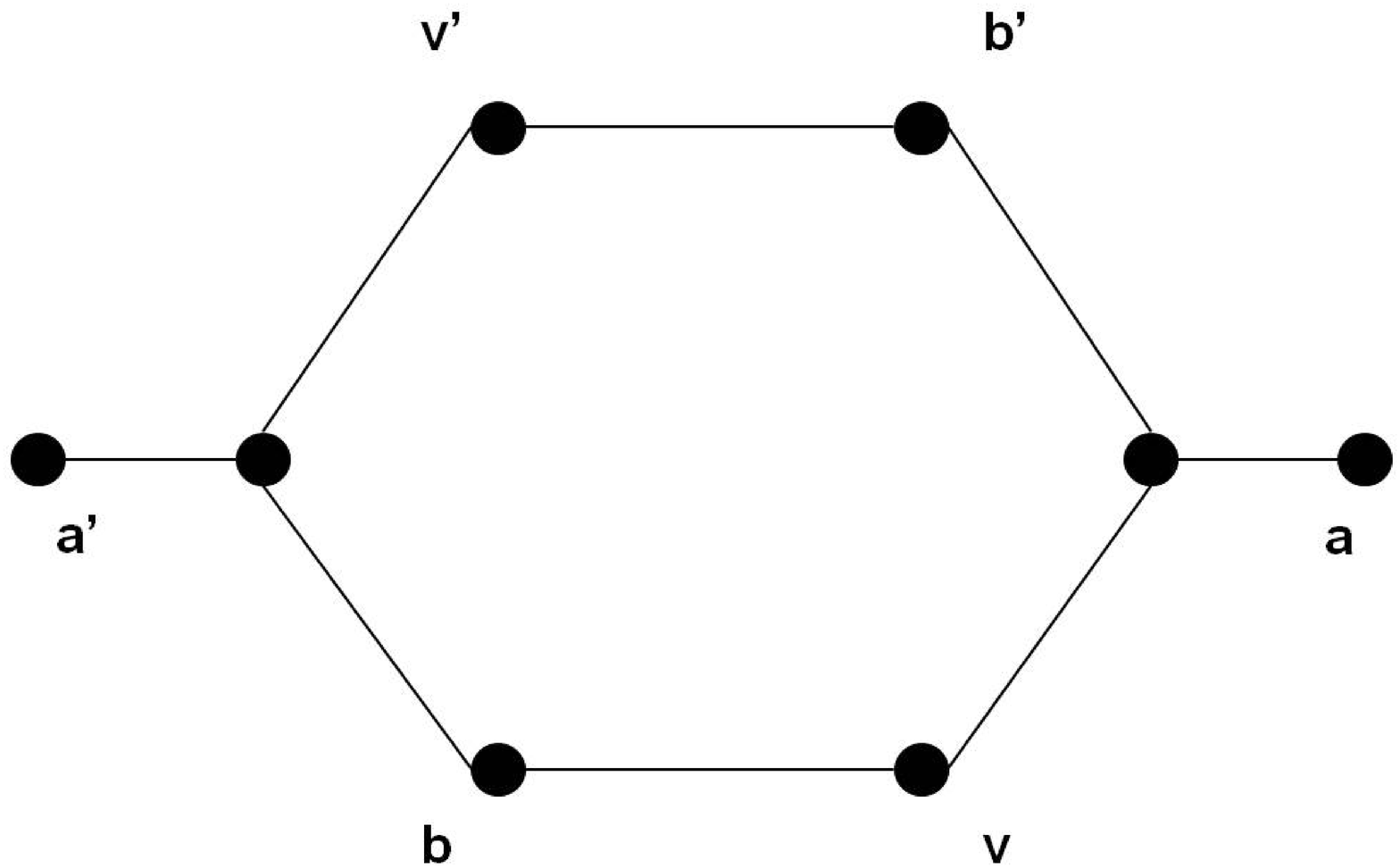}
  \caption{In left hand side graph $i(v)=\{a,v'\}$ and $i'(v)=\{a',v\}$. In right hand side graph $i(v)=\{v'\}$ and $i'(v)=\{v\}$.}\label{grafos-i(v)}
\end{figure}

\begin{corollary}\label{corollaryAntipodal}
Let $G$ be a connected graph of order $n\ge 2$ and let $H$  be a connected 2-antipodal graph of order $n'$. Then
$dim_s(G\circ H) =  \frac{nn'}{2} -1$.
\end{corollary}

\begin{theorem}\label{H is cycle}
Let $C_t$ be a cycle of order $t\ge 3$. For any connected graph $G$ of order $r\ge 2$, $$dim_s(G\circ C_t)=r\left\lceil\frac{t}{2}\right\rceil-1.$$
\end{theorem}

\begin{proof}
Let $V(G)=\{x_1,x_2,...,x_r\}$ and $V(C_t)=\{y_0,y_1,...,y_{t-1}\}$ be the vertex sets of $G$ and $C_t$, respectively. We assume $y_0\sim y_1\sim...\sim y_{t-1}\sim y_0$ in $C_t$ and from now on all the operations with the subscripts of $y_i$ are done modulo $t$. Since $C_t$ is a vertex transitive graph, we can take without loss of generality $v=y_0$ as the root of $C_t$.

If $t$ be an even number, then $C_t$ is 2-antipodal. So the result follows by Corollary \ref{corollaryAntipodal}. Now let $t$ be an odd number. Note that exactly two vertices $y_{\left\lceil\frac{t}{2}\right\rceil}$ and $y_{\left\lfloor\frac{t}{2}\right\rfloor}$ are maximally distant from $v$ in $C_t$. So, from Lemma \ref{maximally_distant with root} we have that every vertex $(x_i,y_l)$ is mutually maximally distant from $(x_j,y_k)$ in $G\circ C_t$, with $j\ne i$ and $l,k\in \left\{\left\lceil\frac{t}{2}\right\rceil,\left\lfloor\frac{t}{2}\right\rfloor\right\}$. Moreover, from Lemma \ref{rooted mutually_maximally_distant} we have that for every $i\in \{1,2,...,r\}$, $(x_i,y_k)$ is mutually maximally distant from $(x_i,y_{k+\left\lfloor\frac{t}{2}\right\rfloor})$ and $(x_i,y_{k+\left\lceil\frac{t}{2}\right\rceil})$ in $G\circ C_t$ with $k\in \{1,2,...,\left\lfloor\frac{t}{2}\right\rfloor-1,\left\lceil\frac{t}{2}\right\rceil+1,...,t-1\}$. Also, the vertex $(x_i,y_{\left\lfloor\frac{t}{2}\right\rfloor})$ is mutually maximally distant from $(x_i,y_{t-1})$ and the vertex $(x_i,y_{\left\lceil\frac{t}{2}\right\rceil})$ is mutually maximally distant from $(x_i,y_1)$. Thus, we obtain that the graph $(G\circ C_t)_{SR}$ is isomorphic to a graph with set of vertices $U\cup \left(\bigcup_{i=1}^{r}V_i\right)$ where $\langle U\rangle$ is isomorphic to a complete $r$-partite graph $K_{2,2,...,2}$ and for every $i\in \{1,...,r\}$, $\langle V_i\rangle$ is isomorphic to a path graph $P_{t-1}$. Notice that the leaves of $P_{t-1}$ belong to $U$, so for every $i\in \{1,...,r\}$, $|V_i\cap U|=2$. Thus, we have the following:
\begin{align*}
  dim_s(G\circ C_t)&=\alpha((G\circ C_t)_{SR})\\
                     &=\alpha(\langle U\rangle)+(r-1)\alpha(P_{t-3})+\alpha(P_{t-1})\\
                     &=2(r-1)+(r-1)\frac{t-3}{2}+\frac{t-1}{2}\\
                     &=r\left\lceil\frac{t}{2}\right\rceil-1.
\end{align*}
The proof is complete.
\end{proof}

We recall that the \emph{clique number} of a graph $H$, denoted by $\omega(H)$, is the number of vertices in a maximum clique in $H$. Two distinct vertices $x$, $y$ are called \emph{true twins} if $N_H[x] = N_H[y]$. We say that $X\subset V(H)$ is a \emph{twin-free clique} in $H$ if the subgraph induced by $X$ is a clique and for every $u,v\in X$ it follows $N_H[u]\ne N_H[v]$, i.e., the subgraph induced by $X$ is a clique and it contains no true twins. We say that the \emph{twin-free clique number} of $H$, denoted by $\varpi(H)$, is the maximum cardinality among all twin-free cliques in $H$. So, $\omega(H)\ge \varpi(H)$.

\begin{theorem}{\em \cite{strongDimensionCorona}}\label{lem_strDimCoronaGen}
Let $G$ be a connected graph of order $r$. Let $H$ be a graph of order $t$ and maximum degree $\Delta$. If $\Delta\le t-2$ or $r\ge 2$, then $dim_s(G\odot H)=rt-\varpi(H).$
\end{theorem}

Given a vertex $v$ of a graph $H$, we denote by $H-v$ the graph obtained by removing $v$ from $H$. Now, if $v$ is a vertex of $H$ of degree $n-1$, then the rooted product graph $G\circ_v H$ is isomorphic to the corona product graph $G\odot (H-v)$. So, as a direct consequence of Theorem \ref{lem_strDimCoronaGen} we obtain the following result.

\begin{corollary}
Let $G$ be a connected graph of order $r\ge 2$. Let $H$ be a connected graph of order $t\ge 2$ and let $v$ be a vertex of $H$ of degree $t-1$. Then $dim_s(G\circ_{v} H)=r(t-1)-\varpi(H-v)$.
\end{corollary}

The next result  gives the exact value for the strong metric dimension of $G\odot H$ when $H$ is a triangle free graph.

\begin{theorem}{\em \cite{strongDimensionCorona}}\label{lem_strDimCoronaTriangleFree}
Let $G$ be a connected graph of order $r$ and let $H$ be a triangle free graph of order $t\ge 3$ and maximum degree $\Delta$. If $r\ge 2$ or $\Delta \le t-2$, then $$dim_s(G\odot H)= rt-2.$$
\end{theorem}

As a direct consequence of Theorem \ref{lem_strDimCoronaTriangleFree} we have the following.

\begin{corollary}
Let $G$ be a connected graph of order $r\ge 2$. Let $H$ be a connected graph of order $t\ge 2$ and let $v$ be a vertex of $H$ of degree $t-1$. If $H-v$ is a triangle free graph. Then $$dim_s(G\circ_{v} H)= r(t-1)-2.$$
\end{corollary}

As the next theorem shows,  the strong metric dimension of $G\odot H$ depends on the diameter of $H$.

\begin{theorem}{\em \cite{strongDimensionCorona}}\label{lem_strDimCoronaDiameter}
Let $G$ be a connected graph of order $r$. Let $H$ be a graph of order $t$ and maximum degree $\Delta$.
\begin{enumerate}[{\rm (i)}]
\item  If $H$ has diameter two and either $\Delta\le t-2$ or $r\ge 2$, then $$dim_s(G\odot H)=(r-1)t+dim_s(H).$$
\item  If $H$ is not connected or its diameter is greater than two, then $$dim_s(G\odot H)=(r-1)t+dim_s(K_1 + H).$$
\end{enumerate}
\end{theorem}

Therefore, as a consequence of Theorem \ref{lem_strDimCoronaDiameter} we obtain the following result for $G\circ_{v} H$.

\begin{corollary}
Let $G$ be a connected graph of order $r\ge 2$. Let $H$ be a graph of order $t\ge 2$ and let $v$ be a vertex of $H$ of degree $t-1$.
\begin{enumerate}[{\rm (i)}]
\item  If $H-v$ has diameter two, then $$dim_s(G\circ_{v} H)=(r-1)(t-1)+dim_s(H-v).$$
\item  If $H-v$ has diameter greater than two, then $$dim_s(G\circ_{v} H)=(r-1)(t-1)+dim_s(H).$$
\end{enumerate}
\end{corollary}

The strong metric dimension of $G\odot H$ depends on the existence or not of true twins in $H$. In this sense, the following result was presented in \cite{strongDimensionCorona}.

\begin{theorem}{\em \cite{strongDimensionCorona}}\label{lem_strDimCoronaTwins}
Let $G$ be a connected graph of order $r$ and let $H$ be a graph of order $t$. Let $c(H)$ be the number of vertices of $H$ having degree $t-1$.
\begin{enumerate}[{\rm (i)}]
\item  If $H$ has no true twins and $r\ge 2$, then $$dim_s(G\odot H) = rt-\omega(H).$$
\item  If the only true twins of $H$ are vertices of degree $t-1$ and $r\ge 2$, then $$dim_s(G\odot H) = rt+c(H)-1-\omega(H).$$
\end{enumerate}
\end{theorem}

Our next result is an interesting consequence of Theorem \ref{lem_strDimCoronaTwins}.

\begin{corollary}\label{prop strDimRootTwins}
Let $G$ be a connected graph of order $r\ge 2$. Let $H$ be a connected graph of order $t\ge 2$ and let $v$ be a vertex of $H$ of degree $t-1$. Let $c(H-v)$ be the number of vertices of $H-v$ having degree $t-2$.
\begin{enumerate}[{\rm (i)}]
\item  If $H-v$ has no true twins, then $$dim_s(G\circ_{v} H) = r(t-1)-\omega(H-v).$$
\item  If the only true twins of $H-v$ are vertices of degree $t-2$, then $$dim_s(G\circ_{v} H) = r(t-1)+c(H-v)-1-\omega(H-v).$$
\end{enumerate}
\end{corollary}

\section{Tight bounds}\label{bounds}

\begin{lemma}\label{Lemmmadivide}
  Let $G$ and $H$ be two connected graphs. Given $x\in V(G)$, $v\in V(H)$ and a strong metric basis $B$ of $G\circ_v H$ let $B_x=B\cap (\{x\}\times V(H))$ and let $M(v)$ be the set of vertices of $ H$ which are maximally distant from $v$. Then the following assertions hold.
  \begin{enumerate}[{\rm (i)}]
   \item   $|B_x|\ge dim_s(H)-1$.

   \item  If  $B_x\supset \{x\}\times M(v)$, then $|B_x|\ge dim_s(H)$.

   \item If $v$ does not belong to any strong metric basis of $H$, then  $|B_x|\ge dim_s(H)$.
   \end{enumerate}
\end{lemma}

\begin{proof}
  First we consider a pair $(x,y),(x,y')$ of adjacent vertices in $(H_{x})_{SR}$, where $y,y'\ne v$. Since $B$ is a vertex cover of $(G\circ_v H)_{SR}$, either  $(x,y)\in B_x$ or  $(x,y')\in B_x$. Thus, $B_x\cup\{(x,v)\}$ is a vertex cover of $(H_{x})_{SR}$.  Note that $(x,v)\not\in \partial(G\circ_{v} H)$ and, as a consequence, $(x,v)\not\in B_x$. Hence, $|B_x|+1=|B_x\cup \{(x,v)\}|\ge dim_s(H_{x})=dim_s(H)$. Therefore, (i) follows.

   Now we suppose $B_x\supset \{x\}\times M(v)$. If $(x,y)$ and $(x,v)$ are adjacent in $(H_{x})_{SR}$, then $y\in M(v)$. So the edge  $\{(x,y),(x,v)\}$ of $(H_{x})_{SR}$ is covered by $(x,y)\in B_x$. Thus, $B_x$ is a vertex cover of $(H_{x})_{SR}$ and, as a result,  $|B_x|\ge dim_s(H)$. Therefore, (ii) follows.

   Finally, suppose that $v$ does not belong to any strong metric basis of $H$.
    Since the function $f: \{x\}\times V(H)\rightarrow V(H)$, where $f(x,y)=y$, is a graph isomorphism  and  $B_x\cup\{(x,v)\}$ is a strong metric generator for $H_{x}$,
    the set $$A=f(B_x\cup\{(x,v)\})=\{v\}\cup \{u: (x,u)\in B_x\}$$ is a strong metric generator for $H$. Thus, since  $v$ does not belong to any strong metric basis of $H$, $|A|>dim_s(H)$. Taking into account that $(x,v)\not\in B_x$  we obtain $|B_x|=|B_x\cup \{(x,v)\}|-1=|A|-1\ge dim_s(H)$.
   The proof is complete.
\end{proof}

\begin{theorem}\label{H is not T v in basis}
Let $G$ be a connected graph of order $n\ge 2$ and let $H$  be a connected graph.
\begin{enumerate}[{\rm (i)}]
\item If   $v\in V(H)$ belongs to a strong metric basis of $H$, then
$$n\cdot dim_s(H)-1\le dim_s(G\circ_{v} H) \le (|\partial(H)|-1)(n-1) + dim_s(H)-1.$$

\item If   $v\in V(H)$ does not belong to any strong metric basis of $H$, then
$$n\cdot dim_s(H)\le dim_s(G\circ_{v} H) \le \left\{\begin{array}{ll}
                            |\partial(H)|(n-1) + dim_s(H), & \mbox{if $v\not \in \partial(H)$,} \\
                             & \\
                             (|\partial(H)|-1)(n-1) + dim_s(H), & \mbox{if $v  \in \partial(H)$.}
                             \end{array} \right.$$
\end{enumerate}
\end{theorem}

\begin{proof}
Let $W$ be a strong metric basis of $H$ such that $v\in W$ and let $B$ be a strong metric basis of $G\circ_{v} H$.  Since $v$ belongs to a metric basis of $H$, we have $v\in \partial(H)$.
Suppose there exists $x\in V(G)$ such that $(x,u)\not \in B_x$  for some $u\in M(v)$.  By Lemma \ref{Lemmmadivide} (i) we obtain  $|B_x|\ge dim_s(H)-1$.
Moreover, by Lemma \ref{maximally_distant with root} we have that for  $x'\in V(G)-\{x\}$ and $u'\in M(v)$ the vertices $(x,u)$ and $(x',u')$ are mutually maximally distant in $G\circ_v H$. Hence, since $(x,u)\not \in B_x$ and  $B$ is  a vertex cover of $(G\circ_{v} H)_{SR}$,  for every $x'\in V(G)-\{x\}$ we have  $B_{x'}\supset \{x'\}\times M(v) $.
So, according to Lemma \ref{Lemmmadivide} (ii) we have $|B_{x'}|\ge dim_s(H)$.
Therefore, $$dim_s(G\circ_{v} H)=|B|=|B_x|+\sum_{x'\in V(G)-\{x\}}|B_{x'}|\ge n\cdot dim_s(H)-1.$$
On the other hand, since $v\in \partial(H)$,  Proposition \ref{boundary-of-rooted} (ii) leads to   $\partial \left(G\circ_{v} H\right)=V(G)\times (\partial(H)-\{v\}).$
We will show that $S=\partial \left(G\circ_{v} H\right) -P$ is a vertex cover for $(G\circ_{v} H)_{SR}$, where $P=\{a\}\times (\partial(H)-W\cup\{v\})$ and $a\in V(G)$.
Let $(x,y)$ and $(x',y')$ be two adjacent vertices in  $(G\circ_{v} H)_{SR}$. If $x\ne a$ or $x'\ne a$, then $(x,y)\in S$ or $(x',y')\in S$. Now let,  $x=x'=a$. Since $H_a\cong H$ and $W$ is a vertex cover for $H$, $\{a\}\times W$ is a vertex cover for $H_a$ and, as a consequence, $(x,y)\in \{a\}\times W \subset S$ or $(x',y')\in \{a\}\times W \subset S$. Hence, $S$ is a vertex cover for $(G\circ_{v} H)_{SR}$. Therefore,
$$dim_s(G\circ_{v} H)\le |S|=(|\partial(H)|-1)(n-1) + dim_s(H)-1.$$
The proof of (i) is complete.

From now on we assume that   $v$ does not belong to any strong metric basis of $H$. The lower bound of (ii) is a direct consequence of  Lemma \ref{Lemmmadivide} (iii). Suppose $v\not\in \partial(H)$. In this case, by Proposition \ref{boundary-of-rooted} (i) we conclude  $\partial \left(G\circ_{v} H\right)=V(G)\times \partial(H).$
By analogy with the proof of the upper bound of (i) we show that $S'=\partial \left(G\circ_{v} H\right) -P'$ is a vertex cover for $(G\circ_{v} H)_{SR}$, where $P'=\{a\}\times (\partial(H)-W')$, $a\in V(G)$ and $W'$ is a strong metric basis of $H$. Hence,
$$dim_s(G\circ_{v} H)\le |S'|=|\partial(H)|(n-1) + dim_s(H).$$

Finally, for the case $v\in \partial(H)$ we have $\partial \left(G\circ_{v} H\right)=V(G)\times (\partial(H)-\{v\})$
and proceeding by  analogy with the proof of the upper bound of (i) we show that
$S''=\partial \left(G\circ_{v} H\right) -P''$ is a vertex cover for $(G\circ_{v} H)_{SR}$, where $P''=\{a\}\times (\partial(H)-W'')$, $a\in V(G)$ and  $W''$ is a strong metric basis of $H$. Thus, in this case
$$dim_s(G\circ_{v} H)\le |S''|=(|\partial(H)|-1)(n-1) + dim_s(H).$$
The proof of (ii) is complete.
\end{proof}

As Corollary \ref{corollari-simplicial} shows, the bounds of Theorem \ref{H is not T v in basis} (i) are tight and the upper bound $dim_s(G\circ_{v} H) \le |\partial(H)|(n-1) + dim_s(H)$ of Theorem \ref{H is not T v in basis} (ii) is tight. To show the tightness of the upper bound  $dim_s(G\circ_{v} H) \le (|\partial(H)|-1)(n-1) + dim_s(H)$  we consider the graph $J$ shown in Figure \ref{graph-J}. Notice that any strong metric basis of $J$ is formed by the vertices $y_2$, $y_4$ and three vertices of the set $\{y_1,y_3,y_5,x_6\}$.

\begin{figure}[h]
  \centering
  \includegraphics[width=0.4\textwidth]{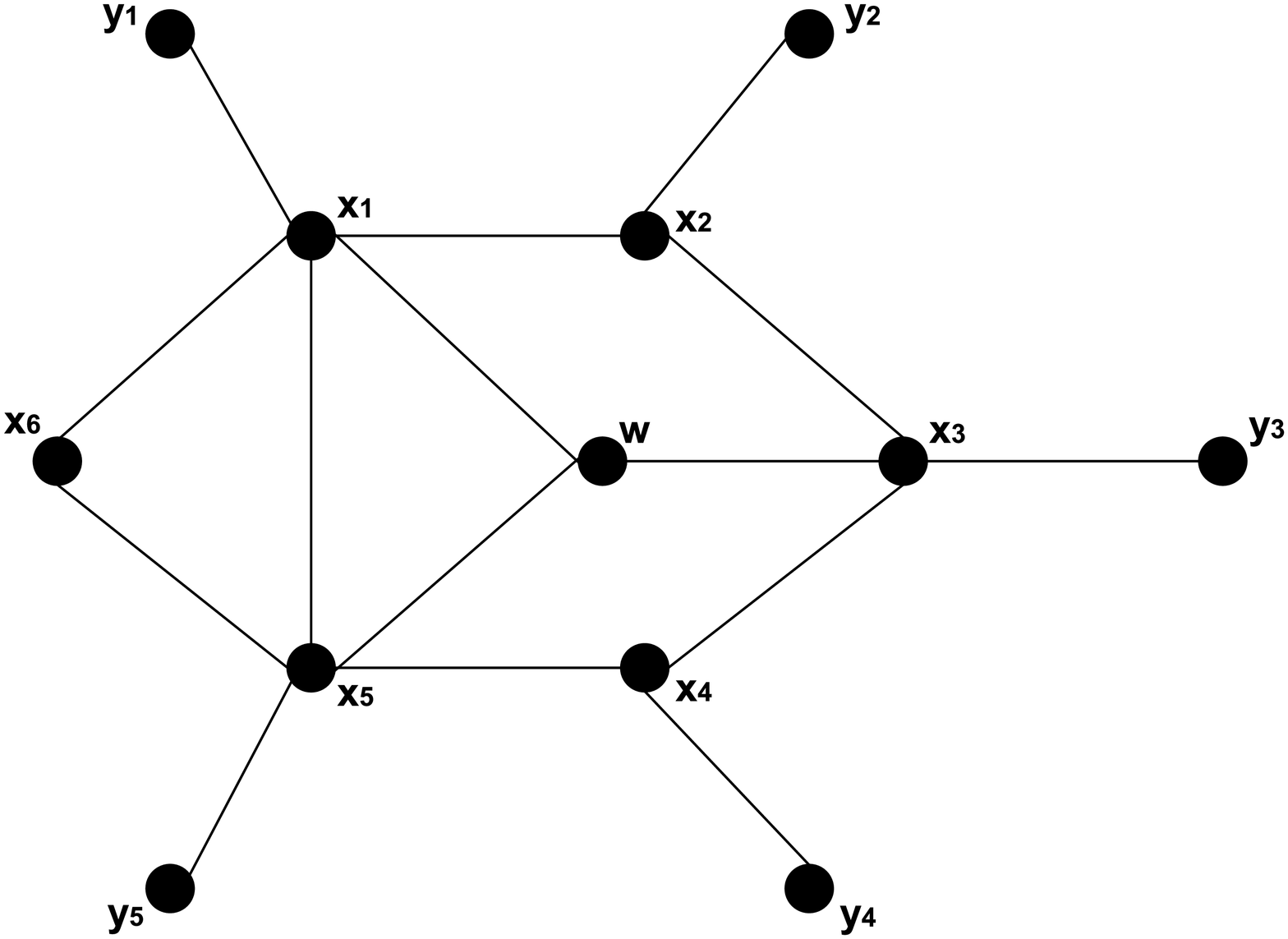}
  \hspace*{2.0cm} \includegraphics[width=0.4\textwidth]{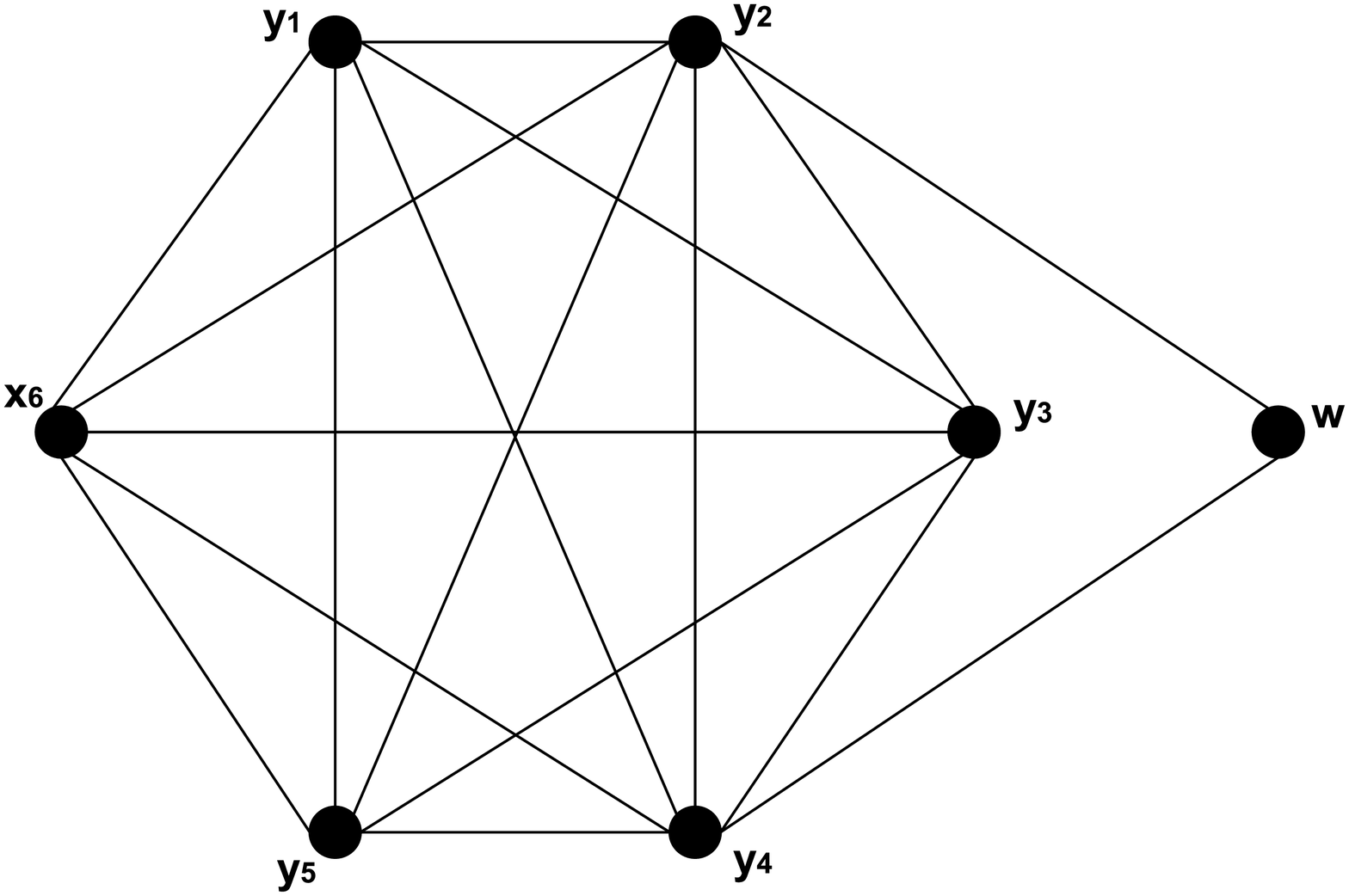}
  \caption{The graph $J$ and its strong resolving graph $J_{SR}$.}\label{graph-J}
\end{figure}

\begin{remark}\label{rem J-omega}
Let $G$ be a connected graph of order $n$. Let $v$ be the vertex of the graph $J$ denoted by $w$. Then $dim_s(G\circ_v J)=(\partial(J)-1)(n-1) +dim_s(J)$.
\end{remark}

\begin{proof}
Let $V=\{u_1,u_2,...,u_n\}$ be the set of vertices of $G$. From Figure \ref{graph-J} we have that there exits six vertices $y_1,y_2,y_3,y_4,y_5$ and $x_6$ which are maximally distant from $v$. So, by using Lemma \ref{maximally_distant with root}, we have that every two vertices $(u_i,y),(u_j,y')\in V\times \{y_1,y_2,y_3,y_4,y_5,x_6\}$, where $i\ne j$, are mutually maximally distant. Moreover, by Lemma \ref{rooted mutually_maximally_distant} for every two mutually maximally distant vertices $z,z'$ in $J$ we have that $(u_i,z),(u_i,z')$ are mutually maximally distant in $G\circ_v J$ for every vertex $u_i$ of $G$. Thus, $(G\circ_v J)_{SR}$ is isomorphic to $K_{6n}$. Therefore, $dim_s(G\circ_v J)=6n-1=(\partial(J)-1)(n-1)+dim_s(J)$.
\end{proof}


To see the tightness of the lower bound of Theorem \ref{H is not T v in basis} (ii) we define the  family  $\mathcal{F}$ of graphs $H$ containing a vertex of degree one not belonging to any strong metric basis of $H$. We begin with the cycle $C_t$, where $t$ is an odd number such that $t\ge 5$, with set of vertices $X=\{x_1,x_2,...,x_t\}$. To obtain a graph $H_{t,p,r}\in \mathcal{F}$ we add the sets of vertices $Y=\{y\}$, $W=\{w_1,w_2,...,w_p\}$ and $Z=\{z_1,z_2,...,z_r\}$, where $p,r\ge 1$, and edges $yx_t$, $x_1x_{t-1}$, $x_{\left \lfloor \frac{t}{2}\right \rfloor}w_i$, for every $i\in \{1,2,...,p\}$, and $x_{\left \lceil \frac{t}{2}\right \rceil}z_j$, for every $j\in \{1,2,...,r\}$. Notice that vertices of $Y\cup W\cup Z$ have degree one in $H_{t,p,r}$ and they are mutually maximally distant between them. Also, for any vertex $a\in N_{H_{t,p,r}}(x_1)$, $d_{H_{t,p,r}}(a,z_j)\le d_{H_{t,p,r}}(x_1,z_j)$, where $j\in \{1,2,...,r\}$. Similarly, for any vertex $b\in N_{H_{t,p,r}}(x_{t-1})$, $d_{H_{t,p,r}}(b,w_i)\le d_{H_{t,p,r}}(x_{t-1},w_i)$, where $i\in \{1,2,...,p\}$. Moreover, we can observe that $x_k$ and $x_{k+\left\lfloor \frac{t}{2}\right\rfloor}$ are mutually maximally distant for every $k\in {2,3,...,\left\lfloor \frac{t}{2}\right\rfloor-1}$. So, $(H_{t,p,r})_{SR}$ is formed by $\left\lfloor \frac{t}{2}\right\rfloor-1$ connected components, that is, $\left\lfloor \frac{t}{2}\right\rfloor-2$ connected components isomorphic to $K_2$ and also, a connected component isomorphic to a graph with set of vertices $Y\cup W\cup Z\cup\{x_1,x_{t-1}\}$ where $\langle Y\cup W\cup Z\rangle$ is isomorphic to $K_{|Y\cup W\cup Z|}$, $x_1$ is adjacent to every vertex $z_j$, $j\in \{1,2,...,r\}$, and $x_{t-1}$ is adjacent to every vertex $w_i$, $i\in \{1,2,...,p\}$. Notice that every $\alpha((H_{t,p,r})_{SR})$-set is formed only by the vertices of $W\cup Z$ and one vertex from each subgraph isomorphic to $K_2$. Therefore, $$dim_s(H_{t,p,r})=\frac{t-5}{2}+p+r$$ and $y$ is a vertex of degree one not belonging to any strong metric basis of $H_{t,p,r}$. The graphs $H_{9,3,4}$ and $(H_{9,3,4})_{SR}$ are shown in Figure \ref{family-F-1}.

\begin{figure}[h]
  \centering
  \includegraphics[width=0.4\textwidth]{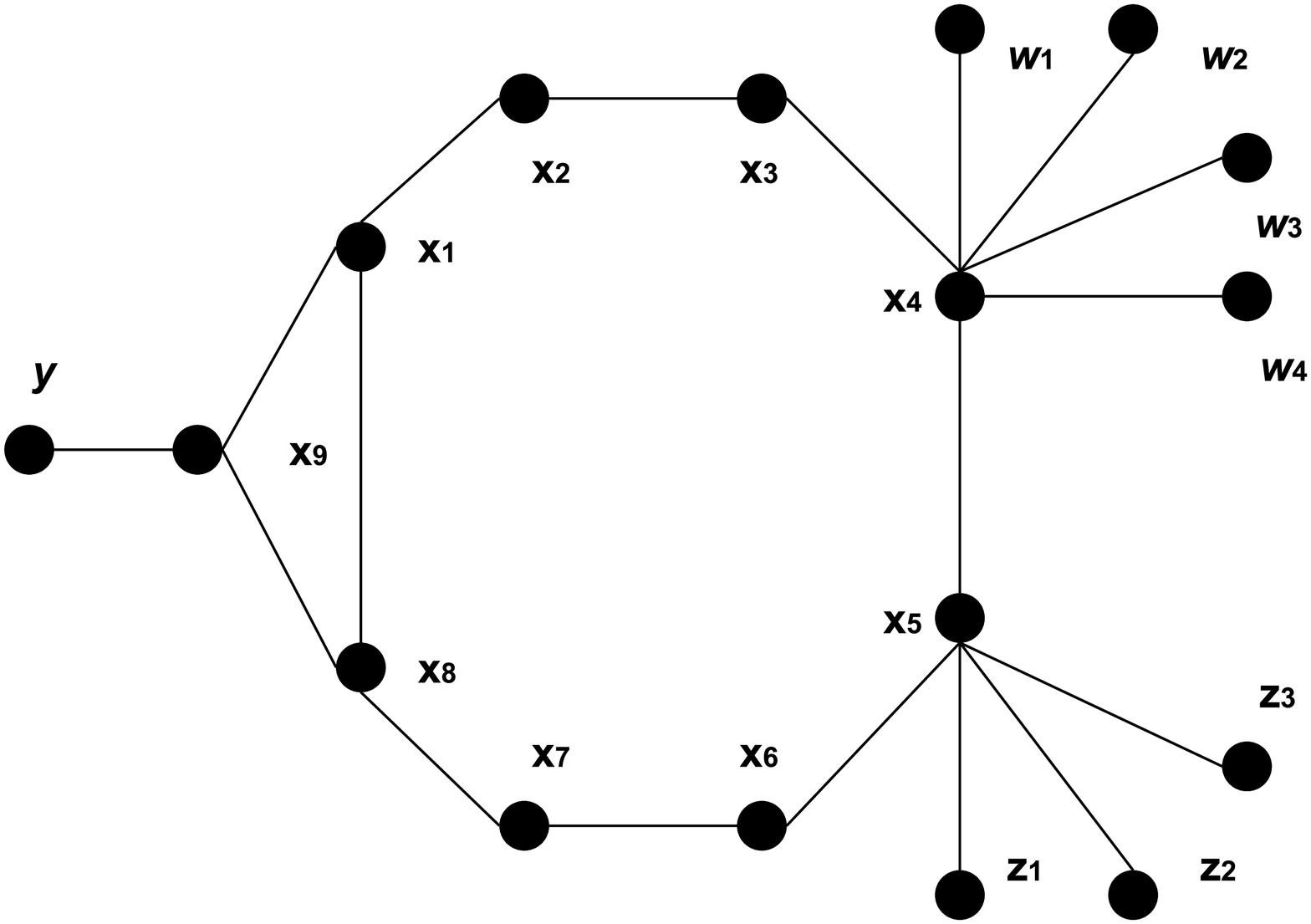}
  \hspace*{2.0cm} \includegraphics[width=0.4\textwidth]{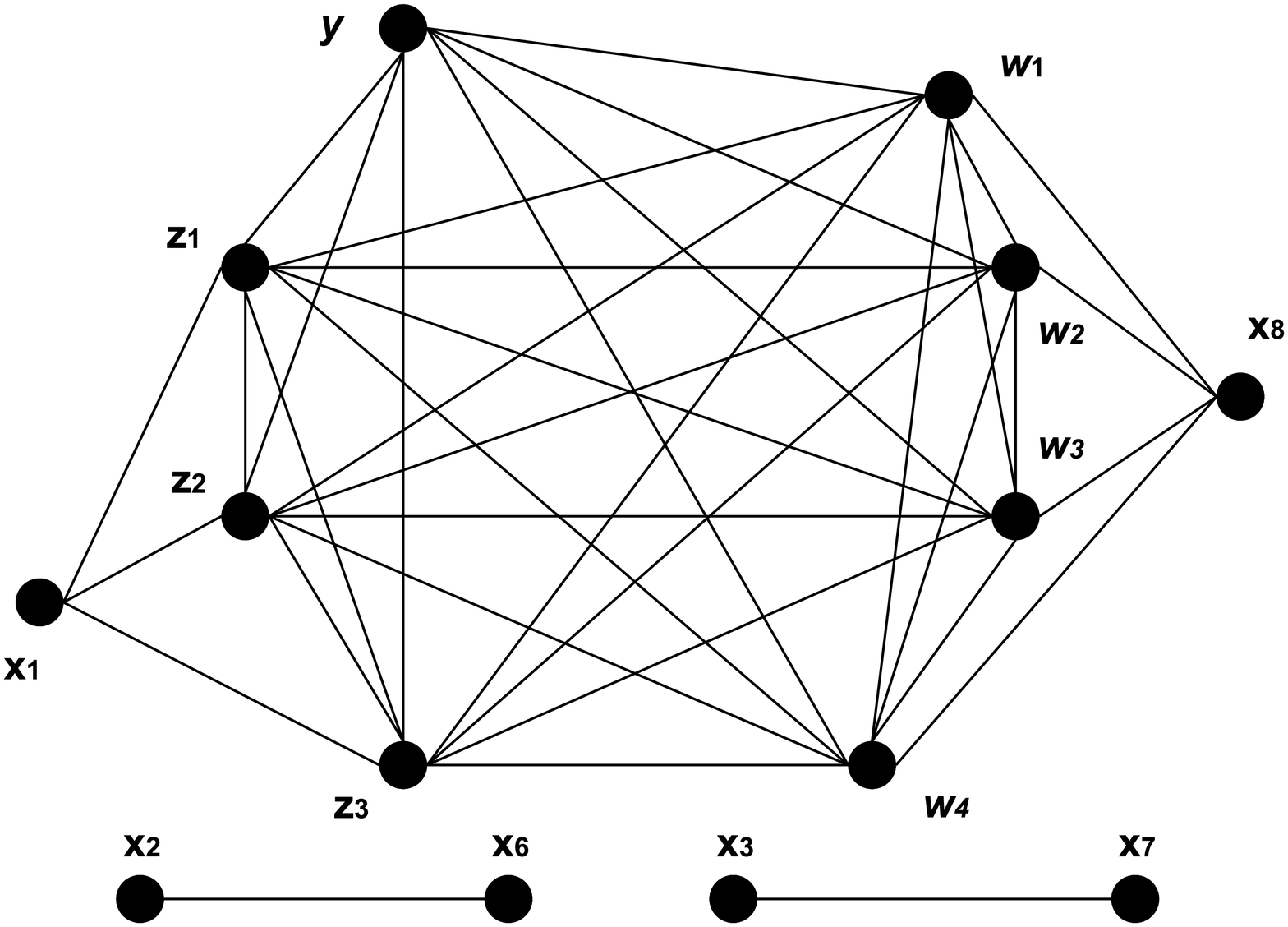}
  \caption{The graphs $H_{9,3,4}$ and $(H_{9,3,4})_{SR}$. The set $S=\{w_1,w_2,w_3,w_4,z_1,z_2,z_3,x_2,x_3\}$ is a strong metric basis of $H_{9,3,4}$.}\label{family-F-1}
\end{figure}

\begin{remark}\label{rem F-1-second}
Let $G$ be a connected graph of order $n$. Let $v$ be the vertex of degree one not belonging to any strong metric basis of the graph $H_{t,p,r}\in \mathcal{F}$. Then
$$dim_s(G\circ_v H_{t,p,r})=n\left(\frac{t-5}{2}+p+r\right)=n\cdot dim_s(H_{t,p,r}).$$
\end{remark}

\begin{proof}
Let $V$ be the vertex set of $G$ and let $H_{t,p,r}\in \mathcal{F}$ with set of vertices $W\cup X\cup Y\cup Z$, where $W=\{w_1,w_2,...,w_p\}$, $X=\{x_1,x_2,...,x_t\}$, $Y=\{y\}$ and $Z=\{z_1,z_2,...,z_r\}$. Since every vertex $u\in W\cup Z$ is maximally distant from $v$, by Lemma \ref{maximally_distant with root}, we have that every two different vertices $(x,y),(x',y')\in V\times (W\cup Z)$, $x\ne x'$, are mutually maximally distant. Moreover, by Lemma \ref{rooted mutually_maximally_distant} for every two mutually maximally distant vertices $v_i,v_j$ in $H_{t,p,r}$ we have that $(u,v_i),(u,v_{j})$ are mutually maximally distant in $G\circ_v H_{t,p,r}$ for every vertex $u$ of $G$. Thus, $(G\circ_v H_{t,p,r})_{SR}$ is formed by $n\frac{t-5}{2}+1$ connected components, i. e., $n\frac{t-5}{2}$ connected components isomorphic to $K_2$ and one connected component isomorphic to a graph $G_1$ with set of vertices $V\times (W\cup Z\cup\{x_1,x_{t-1})\}$ where $\langle V\times (W\cup Z)\rangle$ is isomorphic to $K_{n|W\cup Z|}$ and for every $u\in V$, $(u,x_1)$ is adjacent to every vertex $(u,z_j)$, $j\in \{1,2,...,r\}$, and $(u,x_{t-1})$ is adjacent to every vertex $(u,w_i)$, $i\in \{1,2,...,p\}$. Since in $G_1$ every vertex of $\langle V\times (W\cup Z)\rangle$ has a neighbor not belonging to $V\times (W\cup Z)$ we have that $\alpha(G_1)=n|W\cup Z|$. Therefore, we obtain that
$$dim_s(G\circ_v H_{t,p,r})=\alpha((G\circ_v H_{t,p,r})_{SR})=n|W\cup Z|+n\frac{t-5}{2}=n\left(\frac{t-5}{2}+p+r\right).$$
\end{proof}

According to the Remark \ref{rem F-1-second} we have that for every graph $H\in \mathcal{F}$ and any connected graph $G$ of order $n$, $dim_s(G\circ_{v} H)=n\cdot dim_s(H)$ where $v$ is the vertex of degree one not belonging to any strong metric basis of the graph $H$.

The next result from \cite{strongDimensionCartesian} will be useful to prove Proposition \ref{vertex degree 1}.

\begin{lemma}{\em \cite{strongDimensionCartesian}}\label{lem_simplicial_v}
For every connected graph $G$, $dim_s(G)\ge |\sigma(G)|-1$.
\end{lemma}

\begin{proposition}\label{vertex degree 1}
Let $G$ be a connected graph of order $n\ge 2$ and let $v$ be a vertex of a graph $H$. If $v$ does not belong to the boundary of $H$ and there exists a vertex different from $v$, of degree one in $H$, not belonging to any strong metric basis of $H$, then $$dim_s(G\circ_{v} H)\ge n(dim_s(H)+1)-1.$$
\end{proposition}

\begin{proof}
Let $w$ be a vertex of degree one in $H$ not belonging to any strong metric basis of $H$. Notice that the vertices of the set $A=\{(u_i,w)\;:\;i\in\{1,2,...,n\}\}$ are also vertices of degree one in $G\circ_v H$. Thus, they are simplicial vertices and from Lemma \ref{lem_simplicial_v} we have that at least all but one vertices of $A$ belongs to every strong metric basis of $G\circ_v H$. Thus,
$$dim_s(G\circ_v H)=\alpha((G\circ_v H)_{SR})\ge n\,\alpha(\langle\partial(H)\rangle)+|A|-1=n\,\alpha(H_{SR})+n-1 =n(dim_s(H)+1)-1.$$
\end{proof}

As the following remark shows, the above bound is tight.

\begin{remark}\label{rem F-1}
Let $G$ be a connected graph of order $n$. Let $v$ be the vertex of the graph $H_{t,p,r}\in \mathcal{F}$ adjacent to the vertex of degree one not belonging to any strong metric basis of $H_{t,p,r}$. Then
$$dim_s(G\circ_v H_{t,p,r})=n\left(\frac{t-5}{2}+p+r+1\right)-1=n (dim_s(H_{t,p,r})+1)-1.$$
\end{remark}

\begin{proof}
Let $V$ be the vertex set of $G$. Now, according to the construction of the family $\mathcal{F}$, let the graph $H_{t,p,r}$ with set of vertices $W\cup X\cup Y\cup Z$, where $W=\{w_1,w_2,...,w_p\}$, $X=\{x_1,x_2,...,x_t\}$, $Y=\{y\}$ and $Z=\{z_1,z_2,...,z_r\}$. Since every vertex $y\in W\cup Y\cup Z$ is maximally distant from $v$, by Lemma \ref{maximally_distant with root}, we have that every two different vertices $(x,y),(x',y')\in V\times (W\cup Y\cup Z)$, $x\ne x'$, are mutually maximally distant. Moreover, by Lemma \ref{rooted mutually_maximally_distant} for every two mutually maximally distant vertices $v_i,v_j$ in $H_{t,p,r}$ we have that $(u,v_i),(u,v_{j})$ are mutually maximally distant in $G\circ_v H_{t,p,r}$ for every vertex $u$ of $G$. Thus, $(G\circ_v H_{t,p,r})_{SR}$ is formed by $n\frac{t-5}{2}+1$ connected components, that is, $n\frac{t-5}{2}$ connected components isomorphic to $K_2$ and one connected component isomorphic to a graph $G_1$ with set of vertices $V\times (W\cup Y\cup Z\cup\{x_1,x_{t-1})\}$ where $\langle V\times (W\cup Y\cup Z)\rangle$ is isomorphic to $K_{n|W\cup Y\cup Z|}$ and for every $u\in V$, $(u,x_1)$ is adjacent to every vertex $(u,z_j)$, $j\in \{1,2,...,r\}$, and $(u,x_{t-1})$ is adjacent to every vertex $(u,w_i)$, $i\in \{1,2,...,p\}$. Notice that $\alpha(G_1)=n|W\cup Y\cup Z|-1$. Therefore, we obtain that
$$dim_s(G\circ_v H_{t,p,r})=\alpha((G\circ_v H_{t,p,r})_{SR})=n|W\cup Y\cup Z|-1+n\frac{t-5}{2}=n\left(\frac{t-5}{2}+p+r+1\right)-1.$$
\end{proof}

\end{document}